\documentclass[a4paper, 11pt]{article}

\usepackage[
            includefoot,  %Uncomment to put page number above margin
            marginparwidth=0in,     % Length of section titles
            marginparsep=0in,       % Space between titles and text
            margin=1.45in,               % 1 inch margins
            includemp]{geometry}

\usepackage[T1]{fontenc}
\usepackage[utf8x]{inputenc}
\usepackage{bbm}
\usepackage{float}
\usepackage{graphicx}                  % derni\`ere \'etant la langue principale
\usepackage{amssymb}
\usepackage{amsfonts}
\RequirePackage{amsmath}
\RequirePackage{amsthm}

\usepackage{color}

\usepackage{fancyhdr}

\newtheorem{thm}{Theorem}[section]
\newtheorem{defn}[thm]{Definition}
\newtheorem{prop}[thm]{Proposition}
\newtheorem{lem}[thm]{Lemma}
\newtheorem{rque}{Remark} [section]
\newtheorem{cor}[thm]{Corollary}

\newtheorem{rmq}{Remark}[section]

\newtheorem{assumption}{Assumption}[section]

\DeclareMathOperator{\Hess}{Hess}

\DeclareMathOperator{\Var}{Var}

\DeclareMathOperator{\Ent}{Ent}

\DeclareMathOperator{\argmax}{argmax}
\DeclareMathOperator{\Ch}{Ch}
\DeclareMathOperator{\Lip}{Lip}

\DeclareMathOperator{\aet}{a_{\eta}}
\DeclareMathOperator{\bet}{b_{\eta}}
\newcommand{\R}{\mathbb{R}}

\newcommand{\Q}{\mathfrak{q}}
\DeclareMathSymbol{\Alpha}{\mathalpha}{operators}{"41}\DeclareMathSymbol{\Beta}{\mathalpha}{operators}{"42}\DeclareMathSymbol{\Epsilon}{\mathalpha}{operators}{"45}\DeclareMathSymbol{\Zeta}{\mathalpha}{operators}{"5A}\DeclareMathSymbol{\Eta}{\mathalpha}{operators}{"48}\DeclareMathSymbol{\Iota}{\mathalpha}{operators}{"49}\DeclareMathSymbol{\Kappa}{\mathalpha}{operators}{"4B}\DeclareMathSymbol{\Mu}{\mathalpha}{operators}{"4D}\DeclareMathSymbol{\Nu}{\mathalpha}{operators}{"4E}\DeclareMathSymbol{\Omicron}{\mathalpha}{operators}{"4F}\DeclareMathSymbol{\Rho}{\mathalpha}{operators}{"50}\DeclareMathSymbol{\Tau}{\mathalpha}{operators}{"54}\DeclareMathSymbol{\Chi}{\mathalpha}{operators}{"58}\DeclareMathSymbol{\omicron}{\mathord}{letters}{"6F}

\renewcommand{\bar}{\overline}
\begin{document}

\title{Stability of eigenvalues and observable diameter in RCD$(1, \infty)$ spaces}
\author{Jer\^ome Bertrand and Max Fathi}
\date{\today}

\maketitle

\begin{abstract}
We study stability of the spectral gap and observable diameter for metric-measure spaces satisfying the RCD$(1,\infty)$ condition. We show that if such a space has an almost maximal spectral gap, then it almost contains a Gaussian component, and the Laplacian has eigenvalues that are close to any integers, with dimension-free quantitative bounds. Under the additional assumption that the space admits a needle disintegration, we show that the spectral gap is almost maximal iff the observable diameter is almost maximal, again with quantitative dimension-free bounds. 
\end{abstract}

\section{Introduction}

%{\color{red} en général pour une fonction Lipschitz, $\|nabla f|$ la constante de Lipschitz locale ne coïncide pas nécessairement avec la pente $\nabla f|_*$ utilisée dans la théorie RCD. On a un besoin d'une inégalité de Poincaré locale pour les boules et que la mesure soit doublante.\\
%Max : on peut pas utiliser une telle hypothese, c'est meme pas valable pour la gaussienne. Par contre, on a la Sobolev-to-lipschitz property, qui peut etre appliquée a la fonction définissant les aiguilles.\\
%Jerome: Ok mais cela ne change pas le fait que pour la norme utilisée appliquée à une fonction affine de pente 1, la norme peut être inférieure à 1 }

A classical topic in geometry is to understand the structure of spaces that maximize a given geometric quantity, within a suitable class of spaces. An emblematic example that motivates some of our work here is the Cheeger-Gromoll splitting theorem, which states that manifold with nonnegative Ricci curvature that contain an infinite-length geodesic must split off a line. There are many examples of results with a similar flavor, involving quantities such as the diameter, the spectral gap, isoperimetric-type properties, etc... We shall discuss relevant examples later in this introduction. 

Once such a result is known, it is natural to investigate the structure of spaces that \emph{almost} maximize the geometric quantity, and determine whether they are close in some sense to having the same structure as an extremal space. When investigating classes of compact spaces, a non-explicit stability estimate can often be derived by a compactness argument, via Gromov's precompactness theorem, and showing that a sequence of spaces that saturate in the limit the geometric constraint must converge to an extremal space. See \cite[Theorem 1.5]{CM17a} for an example where this strategy is used. Our work here will deal with a class of non-compact spaces, for which this strategy is unavailable. Moreover, our focus will be on deriving estimates that have a fully explicit dependence on the various parameters (and in particular, making sure the estimates are dimension-free). 

In this work, we shall consider metric measures spaces $(X,d,m)$ which belong to the class of $RCD(1,\infty)$ spaces, normalized so that $m(X)=1$. These spaces are a natural extension of Riemannian spaces satisfying a curvature condition. It is stronger than the Lott-Sturm-Villani curvature dimension condition CD$(1,\infty)$, but more appropriate for the type of questions we investigate here, since it is for example the right setting for the non-smooth generalization of the Cheeger-Gromoll splitting theorem \cite{Gig14}. For some of the results we obtain here, we shall further assume the existence of a $CD(1,\infty)$ disintegration of the metric measure space. Such spaces are studied -for general $K$ and $N$- in \cite{CaMi21} and denoted by $CD^1(K,N)$. These notions will be fully defined in Section \ref{sect_prelim}. 
%This additional property is defined as follows

%\begin{defn}[$CD^1(1,\infty)$] A metric measure space $(X,d,m)$ admits a $CD(1,\infty)$ disintegration if for any $1-$Lipschitz function $\varphi: X \longrightarrow \R$,   there exists a set $Q$ made of geodesics $q$ defined on the interval $X_q \subset \R$, a function $h_q: X_q \longrightarrow \R$ such that the following disintegration holds for any Borel set $B \subset X$:
%$$ m(B) \int_Q \int_{B\cap X_q} \, h_q dt \, d\Q(q),$$
%and for $\Q$-a.e. $q \in Q$, $(X_q, |\cdot|, h_q\, dt)$ is a $CD(1,\infty)$ m.m.s. 
%\end{defn} 
%\begin{rmq} Such spaces are studied -for general $K$ and $N$- in \cite{CaMi21} and denoted by $CD^1(K,N)$.
%\end{rmq}

Examples of metric measures spaces admitting a CD$(1,\infty)$ disintegration are weighted Riemannian manifolds whose Bakry-Emery curvature is greater or equal to one, as proved by Klartag \cite{Kla17}, and $RCD(1,N)$ m.m.s (for any $1<N<+\infty$) as proved by Cavalletti and Mondino in \cite{CM17a}. It seems that an obvious adaptation of the proof \cite[Theorem 3.8]{CM17a} yields that metric measure spaces which are both RCD$(1,\infty)$ and $CD_{loc}(K,N)$, for $K \in \R$ and $1<N <+\infty$, also meet the $CD^1(1,\infty)$ condition.

In an RCD$(1,\infty)$ space, there is a spectral gap $\lambda_1 \geq 1$ for the first eigenvalue of the Laplacian; note that $\lambda_1=1$ holds true for the Gauss space. In \cite{GKKO}, Gigli, Ketterer, Kuwada, and Ohta proved the following rigidity result:

\begin{thm}\label{GKKO}
Let $(X,d,m)$ be an RCD$(1,\infty)$-space, and assume that $\lambda_1=1$. Then there exists an RCD$(1,\infty)$-space $(Y,d_Y,m_Y)$ such that:

\begin{itemize}

\item The metric space $(X,d)$ is isometric to the product space $(\R,|\cdot|)\times (Y,d_Y)$ with the product metric.

\item Through the isometry above, the  measure $m$ coincides with the product measure $e^{-x^2/2}dx\otimes m_Y$. 
\end{itemize} 
\end{thm}

Other rigidity results for RCD$(1,\infty)$ spaces were obtained in \cite{OT20, Han21}. For smooth Riemannian manifolds endowed with their volume measure, the analogous inequality is due to Lichnerowicz and the case of equality -which characterizes the unit sphere- is due to Obata. Obata's theorem has been generalized to RCD$(N-1,N)$-spaces by Ketterer (with $N<\infty)$.

Recall that RCD$(N-1,N)$-spaces have diameter bounded from above by $\pi$ and thus, by Gromov's precompactness theorem, form a compact class of metric measure spaces for the measured Gromov Hausdorff distance. Therefore Ketterer's result also leads to stability results for  RCD$(N-1,N)$-spaces with almost minimal first eigenvalue (or equivalently, almost maximal diameter). In the setting of Riemannian manifolds, the equivalence follows from work of Cheng \cite{Cheng} and Croke \cite{Croke}. No such compactness property is available for the RCD$(1,\infty)$ spaces, which leaves the properties of such spaces with almost minimal first eigenvalue unknown.

Part of our work here will be to study RCD spaces with almost maximal spectral gap, that is RCD$(1,\infty)$ spaces with $\lambda_1 \leq 1 + \epsilon$. Let us recall that an RCD$(1,\infty)$-space has a discrete spectrum \cite[Proposition 6.7]{GMS15}. One of our main results is the following about higher eigenvalues: 

\begin{thm} \label{main_thm_rcd}
Consider an RCD$(1, \infty)$-space, and assume its spectral gap $\lambda_1$ is smaller than $1+\epsilon$ with $\epsilon$ less than some fixed $\epsilon_0$. Then for any $n \geq 1$, $\theta < 1/2$, there is an eigenvalue $-\lambda_n$ of the Laplacian $\Delta$ such that $|\lambda_n - n\lambda_1| \leq C(n, \theta, \epsilon_0)\epsilon^{1/2 - \theta}.$
\end{thm}

Since by assumption $\lambda_1$ is close to one, we hence prove that positive integers are close to being eigenvalues of the Lpalacian when the spectral gap is almost minimal. The integers appear here as the eigenvalues of the Ornstein-Uhlenbeck operator, which is the Laplacian on the Gaussian space. 

In this setting, higher eigenvalues are stable under measured Gromov-Hausdorff convergence \cite{GMS15}. The argument in that work uses a compactness argument, so it is not explicitly quantitative. Moreover, since the eigenfunctions are not necessarily globally Lipschitz, it does not seem like this type of statement can be immediately deduced from a statement on closeness to a Gaussian factor, such as Theorem \ref{thm_gaussian_component_eigen} below. 

In the Euclidean setting, it is known that the sequence of ordered eigenvalues is bounded from below by the ordered eigenvalues of the Ornstein-Uhlenbeck operator on the same space \cite{Mil18}, as a consequence of Caffarelli's contraction theorem. To our knowledge, an analogous result in the full geometric setting is an open problem at the time of writing. 

In the smooth finite-dimensional setting, Petersen \cite{Pet99} proved that an $n$-dimen\-sional manifold with Ricci curvature bounded from below by $n-1$, whose $(n+1)$-th eigenvalue is close to $n$, is close in the Gromov-Hausdorff sense to a sphere. Aubry \cite{Au05} later showed that the statement still holds for the $n$-th eigenvalue, with quantitative bounds, and that a control on the $(n-1)$-th eigenvalue is not enough. More recently, Takatsu \cite{Tak21} proved convergence of the spectral structure of low-dimensional projections of spheres of high-dimension to Gaussian spaces, which is a particular example to which all the results presented here apply. 

We shall also show that the pushforward of the reference measure by a normalized eigenfunction is close to Gaussian, with dimension-free quantitative bounds, partially extending results of \cite{DPF17,CF20} from the Euclidean setting to the full setting of RCD spaces. Similar questions for manifolds with small deficit in the Bakry-Ledoux isoperimetric inequality were raised in \cite[Remark 7.6]{MO19}, and can be positively answered with the tools we use here. 

\begin{thm} \label{thm_gaussian_component_eigen}
Let $(M,d,\mu)$ be an RCD$(1,\infty)$ probability space, and $f$ be a normalized eigenfunction of $-\Delta$ associated to the eigenvalue $\lambda$. Let $\nu=f_{\sharp} \mu$, and $\gamma$ be the standard one-dimensional Gaussian measure. Then,
$$W_1(\nu, \gamma) \leq  4\times 2^{\lambda/2}\sqrt{\lambda -1}.$$
\end{thm}

Here $W_1$ stands for the $L^1$ Wasserstein distance. Asymptotic normality for eigenfunctions on manifolds was studied by E. Meckes \cite{Mec09}, and (a variant of) Theorem \ref{thm_gaussian_component_eigen} can be directly deduced from her abstract theorem via our main Lemma \ref{keylemma} below. When the eigenspace associated with the spectral gap is of dimension higher than one, or if there are several eigenvalues close to one, then a multivariate version of \ref{thm_gaussian_component_eigen} holds, and can be proved with the arguments we develop here, as in the Euclidean setting \cite{CF20}. 

A problem left open by this work is a complete stability estimate for Theorem \ref{GKKO}. We would expect a space with almost minimal spectral  gap to be close in some sense to a product space with Gaussian factor, but we have not been able to prove a satisfactory statement showing an approximate splitting at the level of the metric. 

A second part of our work will deal with stability of the observable diameter, motivated by Ketterer's work on Obata's theorem in RCD spaces, as well as \cite{MO19}. While seeking for upper bound on the diameter of an RCD$(1,\infty)$-space is meaningless, Gromov's observable diameter is a good candidate to replace standard diameter. Let us recall its definition.

\begin{defn}[Observable diameter]
The observable diameter $D_{obs}$ of the metric-measure space $(M, d, \mu)$ is a function of $\kappa \in (0, 1)$ defined by 
$$D_{obs}((M,d,\mu); \kappa) := \sup_{f:M \rightarrow \R }\Big\{ \inf_{E \operatorname{Borel}} \left\{ \operatorname{diam}(E); f_{\sharp}\mu(E) \geq 1-\kappa\right\}; f \; 1\operatorname{-Lipschitz}\Big\}.$$
For ease of notation, the observable diameter is also denoted by $D_{obs}(\mu; \kappa)$.
\end{defn}

As proved in Theorem \ref{ObsCompa} below, the observable diameter of an RCD$(1,\infty)$-space is pointwise bounded from above by that of the Gauss space.

We shall prove some results connecting RCD$(1,\infty)$ spaces with almost minimal spectral gap to those with almost maximal observable diameter. More precisely, we prove

\begin{thm} \label{sg_to_diam}
Let $(X,d,\mu)$ be an RCD$(1,\infty)$ space admitting a $CD(1,\infty)$ disintegration in the sense of Assumption \ref{assumpt_needle}, and whose spectral gap satisfies $\lambda_1 \leq 1 + \epsilon$. Then for any $\theta>0$ small enough and $\kappa\geq C(\theta) \epsilon^{1/20-\theta}$, the following inequalities hold 
$$ D_{obs}( \gamma;\kappa-C(\theta)\epsilon^{1/20 - \theta}) \geq D_{obs}( \mu;\kappa - C(\theta)\epsilon^{1/20 - \theta}) \geq D_{obs}( \gamma;\kappa).$$

\end{thm}

The assumption on the $CD(1,\infty)$ disintegration in this result could be relaxed, in that we only need a disintegration with respect to one particular test function, namely an eigenfunction associated with the spectral gap. 

This theorem can be compared with \cite[Theorem 4.4]{CMS19}, which proved that under a CD$(N-1,N)$ condition if the spectral gap is almost minimal then the diameter is almost maximal, with dimension-dependent exponent and constants.

Finally, we also obtain a converse result, deducing closeness of the spectral gap from closeness of the observable diameter: 

\begin{thm}\label{thm_dobs_to_sg}
Let $(X,d,\mu)$ be an RCD$(1,\infty)$ space admitting a $CD(1,\infty)$ disintegration in the sense of Assumption \ref{assumpt_needle}, and assume that
$$D_{obs}((M,d,\mu); \kappa) \geq D_{obs}((\R, |\cdot|, \gamma); \kappa)     -\epsilon$$
% & \leq   Sep((M,d,\mu); \kappa/2)
for a given $0<\kappa<1$ and $\epsilon>0$ small enough. Then there exists $C = C(\kappa)$ such that the spectral gap $\lambda_1$ satisfies
$$\lambda_1 \leq 1 + C\epsilon^{1/22}.$$

\end{thm}

As a corollary, this implies rigidity of the observable diameter, under the extra Assumption \ref{assumpt_needle}, via the main Theorem of \cite{GKKO}. This statement seems to be new. 

The remainder of the article is as follows: in Section 2, we present some necessary preliminaries on RCD spaces and needle decompositions. In Section 3, we prove Theorems \ref{main_thm_rcd} and \ref{thm_gaussian_component_eigen}, and Section 4 contains the proofs of Theorems \ref{sg_to_diam} and \ref{thm_dobs_to_sg}. 

%%%%%%%%%%%%%%%%%%%%%%%%%%%%%
%%%%%%%%%%%%%%%%%%%%%%%%%%%%%%%
\section{Preliminaries}
\label{sect_prelim}

\subsection{RCD spaces}

\subsubsection{Definitions}

In this work, we shall always work with separable, complete metric spaces, which we shall denote by $(M, d)$. We shall endow it with a Borel probability measure $\mu$. Additionally, we shall assume the metric space is \emph{geodesic}: every pair of points $x, y$ are connected by a minimal geodesic $\gamma : [0, 1] \longrightarrow M$ such that $\gamma(0) = x$, $\gamma(1) = y$ and $d(\gamma(s), \gamma(t)) = |t-s|d(x,y)$ for all $s, t \in [0, 1]$. 

We can endow the space of Borel probability measures on $M$ with finite second order moment, denoted by $\mathcal{P}_2(M)$, with the $L^2$ Wasserstein distance from optimal transport. We refer to the monograph \cite{Vil03} for the definition of $L^p$ Wasserstein distances. 

The relative entropy functional with respect to $\mu$ on $\mathcal{P}_2(M)$ is defined as
$$\Ent_{\mu}(\nu) := \int{\rho \log \rho d\mu}$$
if $\nu = \rho \mu$ is an absolutely continuous probability measure, and $+ \infty$ otherwise. 

With these objects in mind, we can define the curvature-dimension condition CD$(K, \infty)$: 

\begin{defn}
The space $(M, d, \mu)$ is said to satisfy the curvature-dimension condition CD$(K, \infty)$ for $K \in \R$ if the relative entropy is $K$-convex along $W_2$-geodesics on $\mathcal{P}_2(M)$, that is for any pair of probability measures $\nu_1$, $\nu_2$ with finite relative entropy with respect to $\mu$, there is a minimal $W_2$ geodesic $(\nu_t)_{t \in [0,1]}$ connecting them, and such that for all $t \in [0, 1]$ 
$$\Ent_{\mu}(\nu_t) \leq (1-t)\Ent_{\mu}(\nu_0) + t \Ent_{\mu}(\nu_1) -\frac{Kt(1-t)}{2}W_2(\nu_0, \nu_1)^2.$$
\end{defn}

The parameter $K$ plays the role of a Ricci curvature lower bound. Indeed, when the space is a smooth manifold endowed with its (normalized) volume measure, it satisfies this condition iff the Ricci curvature tensor is bounded from below by $K$ times the metric tensor. This definition is a particular instance of the more general curvature dimension CD$(K, N)$, that takes into account the dimension. Here we shall mostly deal with the infinite-dimensional setting, so we omit the full definition, and refer to \cite{Vi09} for the more general setting. 

Given the metric-measure structure, we can also define the Cheeger energy
$$\Ch(f) := \frac{1}{2} \inf_{(f_i)} \liminf \int{|\Lip f_i|^2d\mu}$$
where the infimum runs over the set of all sequences of locally-Lipschitz functions that converge to $f$ in $L^2(\mu)$, and where $\Lip (f)(x)$ is the local Lipschitz constant. %minimal local upper bound on the slope. 

Given an $L^2$ function with finite Cheeger energy, there exists a minimal weak upper gradient $|\nabla f| \in L^2(\mu)$ such that 
$$\Ch(f) = \frac{1}{2}\int{|\nabla f|^2 d\mu}.$$
The Sobolev space $W^{1,2}$ is the space of $L^2$ functions with finite Cheeger energy. We refer to \cite{AGS14b} for more details about these notions. 

\begin{defn}
A metric-measure space $(M, d , \mu)$ that satisfies the CD$(K, \infty)$ condition is said to satisfy the Riemannian curvature-dimension condition RCD$(K, \infty)$ if the Cheeger energy is a quadratic form, that is
$$\Ch(f+ g) + \Ch(f-g) = 2\Ch(f) + 2\Ch(g)$$
for all $f, g \in W^{1,2}$. 
\end{defn}

The most basic example of an RCD$(K, \infty)$ with positive $K$ is a Gaussian measure on $\R$  with variance $K^{-1}$. More generally, a probability measure on $\R^d$ with density $e^{-V}$ w.r.t. the Lebesgue measure satisfies the RCD$(K, \infty)$ condition iff $\Hess V \geq K$. 

On a smooth manifold endowed with its volume measure, the RCD$(K, \infty)$ condition is equivalent to requiring the Ricci curvature tensor to be bounded from below by $K$. More generally, if the measure has density $e^{-V}$ w.r.t. the volume measure, the RCD$(K, \infty)$ condition is equivalent to the Bakry-Emery condition
$$\operatorname{Ric} + \Hess V \geq K\tilde{g}$$
where $\tilde{g}$ is the Riemannian metric tensor of the manifold. However, there are also non-smooth spaces that satisfy the RCD$(K, \infty)$ condition, including Alexandrov spaces \cite{Pet} and some stratified spaces \cite{Ber}. 

On an RCD space, we can then define by polarization the scalar product of two elements of $W^{1,2}$ and the Dirichlet form
$$\langle \nabla f, \nabla g \rangle := \frac{1}{4}(|\nabla(f+g)|^2 - |\nabla (f-g)|^2) \in L^1(\mu); \hspace{2mm} \mathcal{E}(f,g) := \int{\langle \nabla f, \nabla g \rangle d\mu}.$$
The natural analog of the Laplacian on $(M, d, g)$ is the operator $\Delta : D(\Delta) \longrightarrow L^2$ such that
$$\mathcal{E}(f, g) = - \int{g(\Delta f) d\mu} \hspace{2mm} \forall g \in W^{1,2}.$$
The domain $D(\Delta)$ is dense in $W^{1,2}$. We refer to \cite{AGS14} for details of the construction. Note that even when the space is a smooth manifold, $\Delta$ is not simply the Laplace-Beltrami operator, since it takes into account the reference measure $\mu$, that is not necessarily the volume measure. For example, for the Euclidean space endowed with a measure with density $e^{-V}$, the natural Laplace operator is $\Delta_{eucl} - \nabla V \cdot \nabla$, where $\Delta_{eucl}$ is the usual Laplacian on $\R^d$. Finally, note that since the Cheeger energy is used as a Dirichlet form, the induced Laplacian is a local operator, which will allow us to use the diffusion property.

\subsubsection{Properties of RCD$(K, \infty)$ spaces}

An important property of RCD$(K, \infty)$ is the tensorization property, which states that this class of spaces is stable for the product: 

\begin{prop}[Tensorization property]
If two spaces $(M_i, d_i, \mu_i)$ satisfy the RCD$(K, \infty)$ condition, then the product space $(M_1 \times M_2, d_1 \oplus d_2, \mu_1 \otimes \mu_2)$ also does. Here $d_1 \oplus d_2((x_1,x_2), (y_1,y_2)) = \sqrt{d_1(x_1,y_1)^2 + d_2(x_2,y_2)^2}$. 
\end{prop}

In particular this property explains why rigidity statements such as Theorem \ref{GKKO} involve a splitting: the product of an RCD$(1,\infty)$ space with minimal spectral gap and
any other RCD$(1,\infty)$ space is still an RCD$(1,\infty)$ space, and still has minimal spectral gap. 

The spectral gap $\lambda_1$ can be reformulated as the sharp constant in the Poincar\'e inequality for $\mu$, that is the largest constant $C_P$ such that
$$\Var_{\mu}(f)  \leq \frac{1}{C_P}\int{|\nabla f|^2d\mu} \hspace{3mm} \forall f \in W^{1,2}.$$

As we have mentioned in the introduction, an RCD$(K, \infty)$ space with positive $K$ satisfies a Poincar\'e inequality with constant $K^{-1}$, that is $\lambda_1 \geq K$. Moreover, as proved in \cite{GKKO} (with results in the smooth setting obtained in \cite{CZ17}), this bound is rigid: if the sharp spectral gap of an RCD$(K, \infty)$ is equal to $K$, then the space splits off a Gaussian factor that has variance $K^{-1}$. 

The sharp constant is indeed the inverse of the smallest positive eigenvalue $\lambda_1$ of $-\Delta$, that is there exists an eigenfunction $f$ such that
$$\Delta f = -C_P^{-1}f.$$
Existence is ensured by the validity of the logarithmic Sobolev inequality, which is strictly stronger than the Poincar\'e inequality, and provides enough compactness to ensure there is a function that achieves equality in the sharp Poincar\'e inequality, see \cite[Proposition 6.7]{GMS15}. 

%{\color{red} Je crois que finalement on n'utilise pas le r\'esultat suivant.} In general, for an RCD$(K, \infty)$ space, it may be that there are Lipschitz functions for which the minimal weak upper gradient defined from the Cheeger energy does not match with the Lipschitz upper slope. In general, we only have the inequality
%$$|\nabla f|(x) \leq \limsup_{y \rightarrow x} \frac{|f(x)-f(y)|}{d(x,y)}$$
%for locally Lipschitz functions. To bypass this issue in our proofs, we shall use the following result, taken from \cite[Theorem 6.2]{AGS14}: 
%
%\begin{prop}[Sobolev-to-Lipschitz property]
%A function $f \in W^{1,2}$ such that $|\nabla f| \leq C$ almost everywhere admits a $C$-Lipschitz representative (up to sets of measure zero). 
%\end{prop}

\subsection{Localization of the curvature condition}

Part of our work requires an additional assumption on the metric measure space, namely the existence of a suitable needle decomposition. A brief reminder on that decomposition is provided below. This notion originates in Klartag's work \cite{Kla17} in the smooth setting, and was extended to metric measure spaces by Cavalletti and Mondino \cite{CM17a}. Recently, \cite{CGS20} showed that the needle decomposition provides an equivalent way of defining curvature-dimension conditions with finite dimension (which unfortunately is not our setting here).

\begin{thm}[Needle decomposition] \label{thm_needle_dec}
Let $(M, d, \mu)$ be an RCD$(1, N)$ space ($N<+\infty$) and $f$ be a centered, $L^1$ function such that
$$\int{|f(x)|d(x, x_0)d\mu} < \infty.$$
Then the space $X$ can be written as the disjoint union of two measurable sets $Z$ and $T$, where $T$ admits a partition in the following sense. There exists a $1$-Lipschitz function $g$, a partition $(X_q)_{q \in Q}$ of $T$, a probability measure $\Q$ on $Q$ and a family of probability measures $(m_q)_{q \in Q}$, each supported on $X_q$, such that
\begin{enumerate}
\item For any measurable $A$, $\mu(A\cap T) = \int{m_q(A)d\Q(q)}$; 

\item For $\Q$ a.e. $q$, for any $x, y$ in $X_q$, $d(x,y) = |g(x) - g(y)|$, and for any $z \notin X_q$, there exists $x \in X_q$ such that $|g(z) - g(x)| < d(x,z)$; 

\item For $\Q$ a.e. $q$, if $X_q$ is not a singleton then $(X_q, d, m_q)$ is RCD$(1, N)$;

\item For $\Q$ a.e. $q$, $\int_{X_q}{fdm_q} = 0$ and $f=0$ $\mu$-a.e. in $Z$. 
\end{enumerate}
\end{thm}

Note that the properties of the needle decomposition require that $\nu$-almost every needle is not reduced to a singleton, and that $\operatorname{Lip} (g) = 1$ almost everywhere.  

The $1$-Lipschitz function $g$ is called a guiding function, and is obtained as an optimizer in the variational problem

$$\sup_{u 1-lip} \int{f u d\mu}.$$

It can be interpreted as the potential in the Kantorovitch-Rubinstein dual formulation of the $L^1$ optimal transport problem between the positive measures $f_+ \mu$ and $f_-\mu$, which have same mass when $f$ is centered. %We refer to A REF for the analysis of the $L^1$ transport problem on metric spaces. 

Since we consider infinite-dimensional spaces, the current state of knowledge on needle decompositions does not guarantee existence of the decomposition, as the proof uses compactness of finite-dimensional positively curved spaces. However, it may nonetheless exist (indeed, its study in the Euclidean context goes back to work of Payne and Weinberger \cite{PW60}). So we make the following assumption for Theorems \ref{sg_to_diam} and \ref{thm_dobs_to_sg}.  

\begin{assumption} \label{assumpt_needle}
The space $(M, d, \mu)$ satisfies the conclusion of the above theorem with $N = \infty$. 
\end{assumption}

Note that since the second moment is finite and the eigenfunction is $L^2$, the integrability condition $\int{|f|d(x,x_0)d\mu} < \infty$ is automatically satisfied. 

\begin{prop}[Structure of needles] \label{prop_struct_needle}
Without loss of generality, for a.e. $q$ we can identify $(X_q, d, m_q)$ with a space of the form $(\R, |\cdot|, e^{-\psi_q}dx)$ where $\psi''_q \geq 1$ (in the weak sense) and  $\psi_q$ is finite in some interval $I$, and such that, on $X_q$, $g(t) = t + c_q$. 
\end{prop}

\begin{proof}
Let $I = \{g(x), x \in X_q\}$. It is an interval of $\R$. From the structure of the needle, we can parameterize $X_q$ by $I$, and the measure $m_q$ is of the form $e^{-\psi_q}dx$ on $I$. We can then extend $\psi_q$ to take the value $+\infty$ outside $I$, which allows us to identify $(I, |\cdot|, e^{-\psi_q}dx)$ with $(\R, |\cdot|, e^{-\psi_q}dx)$. 
\end{proof}

\subsection{Observable diameter}

We now define the notion of observable diameter, which plays the role of diameter for RCD$(1,\infty)$ spaces, by taking advantage of concentration inequalities. This notion was introduced by Gromov in \cite{Gro}. Informally, a control on the observable diameter states that, even if the diameter itself is not finite, most of the mass is concentrated inside a set of controlled diameter. 

\begin{defn}
The observable diameter $D_{obs}$ of the metric-measure space $(M, d, \mu)$ is a function of $\kappa \in (0, 1)$ defined by 
$$D_{obs}((M,d,\mu); \kappa) := \sup_{f:M \rightarrow \R }\Big\{ \inf_{E \operatorname{Borel}} \left\{ \operatorname{diam}(E); f_{\sharp}\mu(E) \geq 1-\kappa\right\}; f \; 1\operatorname{-Lipschitz}\Big\}.$$
%\sup_{f}\left\{ \inf_{E} \left\{ \operatorname{diam}(f(E)); \mu(E) \geq 1-\kappa\right\}; f : M \longrightarrow \R \; 1\text{-Lipschitz}\right\}.$$

\end{defn}

The observable diameter can be estimated by means of the separation distance whose definition (in the case of two sets with the same lower bound on their mass) is recalled below.

\begin{defn}
The separation distance is a function of $\kappa \in (0, 1)$ defined by 
$$ Sep((M,d,\mu); \kappa):=\sup_{A_1,A_2\subset X \;\rm{Borel}} \left\{ \operatorname{d} (A_1,A_2); \mu(A_i) \geq \kappa, i \in\{1,2\}\right\}.$$

\end{defn}

The following inequality holds (see \cite[Proposition 2.26]{Shioya} for a proof):

\begin{prop} Let $(M, d, \mu)$ be a metric measure space such that $\mu(M)=1$. Then, for any $\kappa \in (0,1)$

\begin{equation}\label{ObsSep}
D_{obs}((M,d,\mu); \kappa) \leq Sep((M,d,\mu); \kappa/2).
\end{equation}

\end{prop}

\begin{rmq}
In the case of the Gaussian space $(\R,|\cdot|,\gamma)$, the symmetry of the density with respect to the origin yields for any $\kappa \in (0,1)$
$$ D_{obs}((\R,|\cdot|,\gamma); \kappa) = Sep((\R,|\cdot|,\gamma); \kappa/2).$$
\end{rmq}

On RCD$(1,\infty)$ spaces, the Lévy-Gromov type isoperimetric inequality has been proved by Ambrosio and Mondino \cite{AmMo16}:

\begin{thm}\label{IIn} Let $(X,d,\mu)$ be an RCD$(1,\infty)$ probability space. Then, for any Borel subset $E\subset X$ it holds
$$ \mathcal{P}(E) \geq \frac{1}{\sqrt{2\pi}} e^{-a_E^2/2},$$
where $\mathcal{P}$ stands for perimeter, and $ \mu(E)= \gamma((-\infty,a_E])$.
\end{thm}

As a consequence, one gets the following comparison inequality for the observable diameter.

\begin{thm}\label{ObsCompa} Let $(M,d,\mu)$ be an RCD$(1,\infty)$ probability space. Then for any $\kappa \in (0,1)$, the following inequalities hold
\begin{equation*}
\begin{split}
 D_{obs}((M,d,\mu); \kappa) & \leq   Sep((M,d,\mu); \kappa/2) \\ & \leq  Sep((\R,|\cdot|,\gamma); \kappa/2) = D_{obs}((\R,|\cdot|,\gamma); \kappa).\end{split}
\end{equation*}

\end{thm}

\begin{proof}
The only property that is yet to prove  is the inequality involving the separation distances of $(\R,|\cdot|,\gamma)$ and $(M,d,\mu)$. The proof is the same as the one of \cite[Lemma 2.32]{Shioya}, we report it since we need the argument later in this work.

Let $A_1,A_2 \subset X$ be two mutually disjoint sets such that $\mu(A_i) \geq \kappa/2$ for $i \in \{1,2\}$. Set $r:= d(A_1,A_2)$ and notice that $A_1 \cap V_r(A_2)=\emptyset$, where $V_r(A_2)$ stands for the open $r$-neighborhood of $A_2$. Let us also define $\sigma(s):= \gamma((-\infty,s])$ and note that $\sigma$ is an increasing function of $s$. Consequently, making use of the isoperimetric inequality for RCD$(1,\infty) $ spaces in integrated form, we infer for any $s>0$
$$ \mu (V_s(A_2)) \geq \gamma((-\infty, -a +\delta +s]) \geq  \gamma((-\infty, -a +s]),$$
where $-a+\delta :=\sigma^{-1}(\mu(A_2)) \geq -a:=\sigma^{-1}(\kappa/2)$. With this notation,
$$ Sep((\R,|\cdot|,\gamma); \kappa/2)=2 |\sigma^{-1}(\kappa/2)| =  2a.$$
Now, observe that
\begin{eqnarray*}
\kappa/2 \leq \mu(A_1)  & \leq &  1 -\mu(V_r(A_2))\\
					 & \leq &  1 - \gamma((-\infty, -a +r]).%\\
					% & \leq & 1-\gamma((-\infty, -\sigma^{-1}(\kappa/2) +r]).
\end{eqnarray*}
Therefore, by symmetry of $\gamma$,
$$ \gamma ((-\infty,-a])  +\gamma(([a - r,+\infty)) = \kappa/2 + \gamma((-\infty, -a +r])\leq  1.$$
Thus, $a-r \geq -a$ hence $2|\sigma^{-1}(\kappa/2)| \geq r$ and the result is proved. 
\end{proof}

\section{Properties of RCD$(1,\infty)$ spaces with almost minimal spectral gap}

In this section, we shall investigate properties of RCD$(1,\infty)$ spaces when the spectral gap is close to one, and prove Theorems \ref{main_thm_rcd} and \ref{thm_gaussian_component_eigen}. 

In the sequel, $L^p$ norms will always be with respect to the reference measure $\mu$, that is 
$$||f||_p := \left(\int{|f|^pd\mu}\right)^{1/p}.$$

\subsection{Main lemma}

We shall first prove the main technical lemma that underlies the proofs of our results, Lemma \ref{keylemma}, which states that if the spectral gap is close to one, then the gradient of an associated eigenfunction is almost constant. 

We shall use the following $L^p$ version of Poincar\'e inequalities: 

\begin{prop} \label{prop_p_poincare}
Under the curvature condition, for any $p \geq 1$we have
$$\left|\left|f - \int{fd\mu}\right|\right|_{p} \leq 2p\left|\left||\nabla f|\right|\right|_{p}.$$ 
\end{prop}

This inequality is strongest for $p = 1$, where it corresponds to the classical Cheeger inequality. See for example \cite[Proposition 2.5]{Mil09} for a proof. The constant in this formulation is not sharp, since for $p=2$ this is not the best estimate on the spectral gap. 

When considering eigenfunctions in positive curvature, we have nice integrability estimates in all $L^p$ spaces, which are summarized in the following proposition: 

\begin{prop}[Integrability of eigenvectors] \label{integ_eigen}
Assume that curvature is bounded from below by $1$. Let $f$ be an eigenfunction of $-{\Delta}$ with eigenvalue $\lambda$ and normalized so that $||f||_2 = 1$. Then
\begin{enumerate}
\item $||f||_p \leq (p-1)^{\lambda/2}$ for any $p \geq 2$; 

\item $|||\nabla f|||_p \leq 2^{\lambda }||f||_{2p}^2 \leq (4p-2)^{\lambda}$ for any $p \geq 1$. 
\end{enumerate}
\end{prop}

\begin{proof}
The proof of (i) can be found in \cite[Section 5.3]{BGL14}. 

To prove (ii), we use the gradient bound  $|\nabla P_tf |^2 \leq (e^{2t} -1)^{-1}P_t(f^2)$ (see \cite[Theorem 4.7.2]{BGL14}). We apply it to the eigenfunction $f$ with $t= (\ln 2)/2$ to get
$$\Gamma(f)^p \leq \left(2^{\lambda }P_t(f^2)\right)^p \leq 2^{\lambda p}P_t(f^{2p})$$
and then integrate with respect to the reference measure $\mu$ to conclude. 
\end{proof}

The key lemma in our proofs is the following result, stating that the gradient of the eigenfunction is almost constant when $\lambda$ is close to one, with quantitative bounds in all $L^p$ norms. A version of this result in the Euclidean setting was derived in \cite{DPF17}, and proved using optimal transport and Caffarelli's contraction theorem. The proof we present here, in addition to being more general, is simpler even in the Euclidean setting. 

\begin{lem}[Key Lemma]\label{keylemma}
Let $f$ be an eigenfunction of $-\Delta$ associated with eigenvalue $\lambda \leq 1+\epsilon$, with $\int{fd\mu} = 0$, $\int{f^2d\mu} = 1$. Then for any $p \in [1, 2[$, we have
$$|||\nabla f|^2 - \lambda||_p \leq 4p(\lambda-1)^{1/2}\lambda^{1/2}\left(\frac{6p-4}{2-p}\right)^{\lambda/2}. $$

Assume now $\lambda \leq \bar{\lambda}$ for some fixed threshold $\bar{\lambda}$. Then for any $\theta \in (0, 1/2)$ and $p \geq 2$ there exists $C = C(p, \theta, \bar{\lambda})$ such that 
$$|||\nabla f|^2 - \lambda||_p \leq C(\lambda - 1)^{1/2 - \theta}.$$
\end{lem}

\begin{proof}

In the non-smooth setting, under the curvature condition we have the following pointwise estimate of \cite[Theorem 3.4]{Sav15} (generalizing a result of \cite{Bak83, Bak94}) that holds almost everywhere: 
\begin{equation}
|\nabla |\nabla f|^2| \leq 2|\nabla f|\sqrt{\Gamma_2(f) - |\nabla f|^2}
\end{equation}
This is only proved for globally Lipschitz $f$ in \cite{Sav15}, but we shall only use its integrated version, so we can proceed by approximation, by taking $h_{n,t} = P_t \max(-n, \min(f, n))$ for $t, n > 0$, which is globally lipschitz (see for example \cite[Proposition 2.5]{GKKO}), and then letting $n$ go to infinity, then using the fact that $f$ is an eigenfunction so $P_t f = e^{-\lambda t}f$ to let $t$ go to zero. In the smooth setting, $|\nabla |\nabla f|^2|^2 \leq 4||\operatorname{Hes} f||_2^2 |\nabla f|^2$. A version of this estimate involving Hessians in the nonsmooth context was investigated in \cite{Gig18}. 

Therefore, using Holder's inequality, ad denoting by $\Gamma_2$ the usual operator
$$\Gamma_2(h) = \frac{1}{2}L\Gamma(h) - \Gamma(h, Lh),$$
we have for $1 \leq p < 2$
\begin{align*}
|||\nabla |\nabla f|^2|||_p^p &\leq 2^p\left(\int{(\Gamma_2(f) - |\nabla f|^2)d\mu}\right)^{p/2}\left(\int{|\nabla f|^{2p/(2-p)}d\mu}\right)^{(2-p)/2} \\
&\leq 2^p(\lambda-1)^{p/2}\left(\int{|\nabla f|^2d\mu}\right)^{p/2}\left(\int{|\nabla f|^{2p/(2-p)}d\mu}\right)^{(2-p)/2} \\
&\leq 2^p(\lambda-1)^{p/2}\lambda^{p/2}\left(\frac{6p-4}{2-p}\right)^{\lambda p/2}
\end{align*}
To conclude, we just apply the $L^p$ Poincar\'e inequality to $|\nabla f|^2$. 

To obtain the inequality for $p \geq 2$, using Holder's inequality we have
$$|||\nabla f|^2 - \lambda||_p \leq |||\nabla f|^2-\lambda||_1^{1-\theta}|||\nabla f|^2-\lambda||_{p'}^{\theta} \leq |||\nabla f|^2-\lambda||_1^{1-\theta}(\lambda + |||\nabla f|^2||_{p'})^{\theta}$$
with $p'$ suitably large enough, and we then apply Proposition \ref{integ_eigen}. 
\end{proof}

%%%%%%%%%%%%%%%%%%%%%%%%%%%%%
%%%%%%%%%%%%%%%%%%%%%%%%%%%%%
\subsection{Spectral properties}

%The following lemma is useful if we assume the spectrum is discrete. A result of F.-Y. Wang \cite{Wan00} implies it is true here, via the LSI, but the proof should be checked (does it apply to general spaces, or just smooth ones with finite dimension...)

Let us start with the following general lemma: 

\begin{lem} \label{perturbed_eigen} Given $\mathcal{L}$ an unbounded operator on $L^2(\mu)$ with discrete spectrum. Assume that there is a function $f$ such that $\mathcal{L}f = \alpha f + g$ with $||f||_2 = 1$ and $||g||_2 \leq \epsilon$. Then there is an eigenvalue of $\mathcal{L}$ at distance at most $\epsilon$ of $\alpha$. 
\end{lem}

\begin{proof}
Let $\alpha_i$ be the collection of eigenvalues of $\mathcal{L}$, with associated orthonormal basis of eigenfunctions $f_i$. We can write $f = \sum \beta_i f_i$ and $g = \sum \delta_if_i$ with $\sum \beta_i^2 = 1$ and $\sum \delta_i^2 \leq \epsilon^2$. Then 
$$\sum (\alpha_i -\alpha)\beta_if_i =  \sum \delta_i f_i$$
and hence for any $i$, $(\alpha_i -\alpha)\beta_i = \delta_i$, so that $\sum (\alpha_i -\alpha)^2\beta_i^2 \leq \epsilon^2$. But since $\sum \beta_i^2 = 1$, there is at least one index $i$ such that $|\alpha_i - \alpha| \leq \epsilon$. 
\end{proof}

\smallskip

With all this in hand, we can now prove Theorem \ref{main_thm_rcd}: 

\begin{proof}[Proof of Theorem \ref{main_thm_rcd}]
Consider $H_n$ the $n$-th Hermitte polynomial, defined via $H_0 = 1$, $H_1(x) = x$ and $H_{n+1}(x) = xH_n(x) - nH_{n-1}(x)$. Recall that $H_n''(x) - xH_n'(x) = -nH_n(x)$, and that they form an orthogonal basis of eigenfunctions for the one-dimensional Ornstein-Uhlenbeck generator.

Let $f$ be a normalized eigenfunction of $-\Delta$ with eigenvalue $\lambda$. From the diffusion property \cite[Section 2.2]{Sav15}, 
\begin{equation}
{\Delta}H_n(f) = H_n'(f){L}f + H_n''(f)|\nabla f|^2 = -\lambda n H_n(f) + H_n''(f)(|\nabla f|^2 - \lambda). 
\end{equation}

According to Lemma \ref{keylemma},   we have a suitable $L^p$ bound on $|\nabla f|^2 - \lambda$ for all $p$, and $H_n''(f)$ has controlled $L^p$ norm via Proposition \ref{integ_eigen}, we control the $L^2$ norm of $H_n''(f)(|\nabla f|^2 - \lambda)$ by some $C(n, \theta, \bar{\lambda})(\lambda-1)^{1/2-\theta}$, and can then apply Lemma \ref{perturbed_eigen} to conclude. 
\end{proof}

\subsection{Proof of Theorem \ref{thm_gaussian_component_eigen}}

We shall now prove that the pushforward of the measure by the eigenfunction is approximately Gaussian, generalizing part of the main result of \cite{CF20} to the RCD setting. The proof will be based on (a version of) Stein's lemma, a classical result in probability and statistics, which allows to control distances to a Gaussian target measure via deficits in an integration-by parts formula. This was already the strategy used in \cite{Mec09,CF20}. We refer to \cite{Ros11} for the proof, and an introduction to Stein's method. We could also proceed by directly combining the main abstract theorem of \cite{Mec09} with Lemma \ref{keylemma}. 

\begin{lem}
Let $\gamma$ be the standard Gaussian measure on $\R$. Then for any $\mu \in \mathcal{P}(\R)$, we have
$$\max\left( d_{TV}(\mu, \gamma), W_1(\mu, \gamma) \right)\leq \sup_{||f||_{\infty}, ||f'||_{\infty} \leq 4} \int{f' - xf d\mu}. $$
\end{lem}

Here $W_1$ stands for the $L^1$ optimal transport distance, and $d_{TV}$ for the total variation distance on $\mathcal{P}(M)$. 

To keep the statement compact, we have not given the sharp required bounds on $f$ and $f'$ in the integration by parts formula. If we only wish to control $W_1$, we can additionally assume the second derivative to be bounded. Higher-dimensional versions of these estimates are available, although typically getting good estimates on the total variation distance in higher dimension is much harder than estimates on $W_1$, due to regularity issues for solutions to certain Poisson equations.

Let $g : \R \rightarrow \R$ be a $1$-Lipschitz function. According to Stein's lemma,
$$W_1(\nu, \gamma) \leq 4\sup_{\tiny g \,1\mbox{-Lipschitz}} \int{(g' - xg)d\nu}.$$
Actually, the factor $4$ is not necessary, if one uses a version of Stein's lemma with sharper constants. 
We have
\begin{align*}
\int{xg(x)d\nu} &= \int{f(x)g(f(x))d\mu} = -\lambda^{-1}\int{Lf(x) g(f(x))d\mu} \\
&= \lambda^{-1}\int{\Gamma(f, g \circ f)(x)d\mu} = \lambda^{-1}\int{g'(f(x))\Gamma(f)(x)d\mu}
\end{align*}

Hence, using Lemma \ref{keylemma}, we get 
\begin{align*}
\sup_{\tiny g \,1\mbox{-Lipschitz}} \int{(g' - xg)d\nu} &\leq  \lambda^{-1}\int{|\Gamma(f)-\lambda|d\mu} \\
&\leq  4\times 2^{\lambda/2}\sqrt{\lambda -1}. 
\end{align*}

%%%%%%%%%%%%%%%%%%%%%%%%%%%%%%%%%%%%%%%%
%%%%%%%%%%%%%%%%%%%%%%%%%%%%%%%%%%%%%%%%%%%

\section{Spectral gap and observable diameter}

This section aims at proving that for an RCD$(1,\infty)$ space that admits $CD(1,\infty)$ disintegration, the spectral gap is close to that of the Gaussian space if and only if the observable diameter is close to that of the Gauss space.

\subsection{Proof of Theorem \ref{sg_to_diam}}

Let $f$ be a normalized eigenfunction of $-\Delta$, with eigenvalue $\lambda_1$, as in the previous section. We consider the guiding function in the needle decomposition 
\begin{equation}\label{gf} 
g := \underset{\tiny u \,1\mbox{-Lipschitz}}{\argmax}  \hspace{1mm} \int{uf d\mu}.
\end{equation}
Without  loss of generality, we may assume that $g$ is centered. Note that since $f$ is an eigenfunction, $g$ also maximizes $\int{\nabla u \cdot \nabla fd\mu}$.% Moreover, $|\nabla g| = 1$ a.e.

The approach follows in the spirit the same one as in \cite{MO19}. In a first lemma, we estimate the $L^2$-norm of $f$ and $|\nabla f|$ along many needles.

\begin{lem}\label{EstNeed1}
Assume $\lambda_1 \leq 1 + \epsilon$ with $\epsilon \leq 1$. There is a set of needles $Q_{\delta}$ such that $\Q(Q_{\delta}) \geq 1 - \delta$ and for any needle $q \in Q_{\delta}$ we have
$$1 - \frac{48\sqrt{\epsilon} + 2\epsilon}{\delta} \leq \int{f^2dm_q} \leq 1 + \frac{48\sqrt{\epsilon}}{\delta}$$
and
$$\int{f^2dm_q} \leq \int{|\nabla f|^2dm_q} \leq \int{f^2dm_q} + \frac{2\epsilon}{\delta}.$$
\end{lem}

\begin{proof}

On a.e. needle $q$, we have by the Poincar\'e inequality on the needle
$$\int{f^2dm_q} \leq \int{|f'|^2dm_q} \leq \int{|\nabla f|^2dm_q}.$$
Therefore, by Markov's inequality, we have
\begin{eqnarray*}
\Q\left(\left\{q; \int{|\nabla f|^2dm_q} \geq \int{f^2dm_q} + 2\epsilon/\delta\right\}\right)& \leq &\frac{\delta\big(\iint{|\nabla f|^2dm_qd\Q(q)} - \iint{f^2dm_qd\Q(q)\big)}}{2\epsilon} \\ &\leq & \delta/2.\end{eqnarray*}
So for an arbitrarily large proportion of needles, we almost have equality in Poincar\'e's inequality for $f$, up to a deficit of order $\epsilon$. 

Then from Lemma \ref{keylemma} with $p=1$, 
$$\int{\left|\int{(|\nabla f|^2 - 1)dm_q}\right| d\Q(q)} \leq \int{||\nabla f|^2 - 1|d\mu} \leq 24\sqrt{\epsilon}.$$ 
Hence 
$$\Q\left(\left\{q; \left|\int{|\nabla f|^2dm_q} - 1\right| \geq 48\sqrt{\epsilon}/\delta\right\}\right) \leq \delta/2.$$
Taking the intersection, we get a set of needles of mass at least $1 - \delta$ on which 
$$1-\frac{48\sqrt{\epsilon}}{\delta} \leq \int{|\nabla f|^2dm_q} \leq 1 + \frac{48\sqrt{\epsilon}}{\delta}$$
and
$$\int{f^2dm_q} \leq \int{|\nabla f|^2dm_q} \leq \int{f^2dm_q} + \frac{2\epsilon}{\delta}.$$
Reordering these bounds leads to the desired result. 
\end{proof}

Next, we prove the $L^2$-closeness of $f$ and $g$ for a large set of needles.
\begin{lem}\label{fg-clo}
For any $q \in Q_{\delta}$, there exists a constant $c_q$ such that
$$\int{(f - g - c_q)^2dm_q} \leq \frac{C\sqrt{\epsilon}}{\sqrt{\delta}}.$$
\end{lem}

\begin{rmq}
Recall that for $\Q$-a.e. $q$, $\int f \, dm_q=0$. Since the average is an $L^2$ projection, we can therefore replace $c_q$ by $\int{gdm_q}$ in the above estimate.  Moreover, we can infer from the lemma above that
\begin{equation}\label{cq-mean}
\left| \int g \, dm_q -c_q\right| \leq  \frac{C\epsilon^{1/4}}{\delta^{1/4}}.
\end{equation}
\end{rmq}

\begin{proof}
Let $\Theta = \left(\int{f^2dm_q}\right)^{-1/2}$. Recall that we can identify each needle $X_q$ with the space $( \R, |\cdot|, e^{-\psi_q}dt)$, by Proposition \ref{prop_struct_needle}.  Using the results of \cite{CF20} on stability of the Bakry-Emery theorem on an Euclidean space, and since $\int{fdm_q} = 0$, we have
$$\int_{\R}{(\Theta f- t)^2e^{-\psi_q(t)}dt} \leq \frac{\sqrt{18\epsilon}}{\sqrt{\delta}}$$
and therefore, from the parametrization of $g$ on the needle,
$$\int_{X_q}{(\Theta f- g - c_q)^2e^{-\psi_q(t)}dt} \leq \frac{\sqrt{18\epsilon}}{\sqrt{\delta}}.$$
From Lemma \ref{EstNeed1}, and provided $\sqrt{\epsilon}/\delta\leq C_0 $, we have $|\Theta ^2 - 1| \leq C\sqrt{\epsilon}/\delta$; therefore
$$|\Theta - 1|^2 \leq C\epsilon/\delta^2.$$
Then 
\begin{eqnarray*}
\int_{\R}{( f- g - c_q)^2e^{-\psi_q(t)}dt} &\leq & 2(\Theta -1)^2\int_{\R}{f^2} e^{-\psi_q(t)}dt  \\ &{} &\;\hspace{2,3cm} +     2\int_{\R}{(\Theta f- g - c_q)^2e^{-\psi_q(t)}dt}, \\
&\leq &  C\sqrt{\epsilon}/\sqrt{\delta}.
\end{eqnarray*}
\end{proof}

In the last technical lemma of this section, we show that $g$ is almost centered on a large set of needles. 
\begin{lem}\label{teclem2}
There exists a set of needles $\bar{Q}_{\alpha}$ with $\Q(\bar{Q}_{\alpha}) \geq 1 - \alpha$ such that for any $q \in \bar{Q}_{\alpha}$ we have 
$$\left|\int{g\,dm_q}\right|^2 \leq C\epsilon^{1/5}/\alpha$$ 
\end{lem}

\begin{proof}
Recall that we assumed $g$ is centered with respect to $\mu$. Therefore, since $g$ is $1$-Lipschitz, the $L^2$ version of the Poincaré inequality yields 
$$1 \geq \Var_{\mu}(g) = \int{\Var_{m_q}(g)d\Q(q)} + \int{\left(\int{gdm_q}\right)^2d\Q(q)}.$$
If $q \in Q_{\delta}$ (which is still the set provided by Lemma \ref{EstNeed1}) then 
$$\Var_{m_q}(g) \geq \left(\Var_{m_q}(f)^{1/2} - \Var_{m_q}(f-g)^{1/2}\right)^2 \geq 1 - C\epsilon^{1/4}/\delta^{1/4},$$
where we use Lemma \ref{fg-clo} and \eqref{cq-mean} to get the inequality.
Consequently, we obtain
$$\int{\Var_{m_q}(g)d\Q(q)} \geq \Q({Q}_{\delta})(1 - C\epsilon^{1/4}/\delta^{1/4}) \geq (1-\delta)(1 - C\epsilon^{1/4}/\delta^{1/4}),$$
so that
$$\int{\left(\int{gdm_q}\right)^2d\Q(q)} \leq 1 - (1-\delta)(1 - C\epsilon^{1/4}/\delta^{1/4}) \leq \delta + C\epsilon^{1/4}/\delta^{1/4}.$$
Taking $\delta$ of order $\epsilon^{1/5}$ and using Markov's inequality yields the result. 
\end{proof}

%%%%%%%%%%%%%

We finally prove the $H^1$-closeness of $f$ and $g$ provided that the spectral gap is close to $1$.

\begin{thm} Let $(M,d,\mu)$ be an RCD$(1,\infty)$ space admitting CD$(1,\infty)$ disintegration. Assume $\lambda_1 \leq 1 + \epsilon$ with $\epsilon < 1$. For any $\theta$ small enough, there exists $C(\theta) > 0$ such that
$$\int{(f-g)^2d\mu} \leq C(\theta)\epsilon^{1/10 - \theta}$$
and
$$\int{|\nabla f - \nabla g|^2d\mu} \leq C(\theta)\epsilon^{1/20 - \theta}.$$
\end{thm}

\begin{proof}
Let us decompose $\|f-g\|_{L^2}^2$ as follows:
$$\int{(f-g)^2d\mu} = \int_{Q_{\delta}\cap\bar{Q}_{\alpha}}{\int{(f-g)^2dm_q} d\Q(q)} + \int_{(Q_{\delta}\cap\bar{Q}_{\alpha})^c}{\int{(f-g)^2dm_q} d\Q(q)}.$$
To estimate the first term on the right hand side, we build on Lemmas \ref{fg-clo} and \ref{teclem2}:
\begin{align*}
&\int_{Q_{\delta}\cap\bar{Q}_{\alpha}}{\int{(f-g)^2dm_q} d\Q(q)} \\
&= \int_{Q_{\delta}\cap\bar{Q}_{\alpha}}{Var_{m_q}(f-g) d\Q(q)} + \int_{Q_{\delta}\cap\bar{Q}_{\alpha}}{\left(\int{f-g dm_q}\right)^2 d\Q(q)}\\
&\leq \int_{Q_{\delta}\cap\bar{Q}_{\alpha}}{\int{(f-g - c_q)^2dm_q} d\Q(q)} + \int_{Q_{\delta}\cap\bar{Q}_{\alpha}}{\left(\int{g dm_q}\right)^2 d\Q(q)} \\
&\leq C\sqrt{\epsilon}/\sqrt{\delta} + C\epsilon^{1/5}/\alpha.
\end{align*}
In order to estimate the second term, we set $p$ and $q$ two positive numbers with $p^{-1} + q^{-1} = 1$.
\begin{align*}
&\int_{(Q_{\delta}\cap\bar{Q}_{\alpha})^c}{\int{(f-g)^2dm_q} d\Q(q)} \leq \Q((Q_{\delta}\cap\bar{Q}_{\alpha})^c)^{1/p}\left(\int{(f-g)^{2q}d\mu}\right)^{1/q} \\
&\leq 2(\alpha + \delta)^{1/p}\left(\left(\int{f^{2q}d\mu}\right)^{1/q} + \left(\int{g^{2q}d\mu}\right)^{1/q}\right)  \\
&\leq 2(\alpha + \delta)^{1/p}((2q-1)^{1 + \epsilon} + 16q^2) \\
&\leq Cq^2(\alpha + \delta)^{1/p}.
\end{align*}
Here we used Proposition \ref{integ_eigen} to control the $L^{2q}$ norm of $f$ and Proposition \ref{prop_p_poincare} to control the $L^{2q}$ norm of $g$. 

Combining the two estimates yields
$$\int{(f-g)^2d\mu} \leq C(q^2(\alpha + \delta)^{(q-1)/q} + \sqrt{\epsilon}/\sqrt{\delta} + C\epsilon^{1/5}/\alpha).$$
We then take $\alpha$ of order $\epsilon^{1/10}$, $q$ large enough and $\delta$ of order say $\epsilon^{3/10}$ (the influence of $\delta$ is on lower order terms anyway). 

For the second estimate, we expand the norm of the gradients as follows:
\begin{align*}
\int{|\nabla f - \nabla g|^2d\mu} &= \int{|\nabla f|^2 + |\nabla g|^2 - 2\nabla f \cdot \nabla g d\mu} \\
&\leq 2 + \epsilon + 2 \int{(Lf)gd\mu} \\
&= 2 + \epsilon -2  \lambda_1 \int{fgd\mu}  \\
&= 2 + \epsilon -2\lambda_1 \int{f^2d\mu} - 2\lambda_1 \int{f(g-f)d\mu} \\
&\leq \epsilon + 2\lambda_1||f||_{L^2}||f-g||_{L^2} \\
&\leq C(\theta)\epsilon^{1/20 - \theta}. 
\end{align*}
\end{proof}

%\begin{cor}
%If $\lambda_1 \leq 1 + \epsilon$ with $\epsilon \leq 1$ then $D_{L^2} \geq 1 - C(\theta)\epsilon^{1/20-\theta}$.
%\end{cor}

Proceeding as in the proof of Theorem \ref{thm_gaussian_component_eigen}, we also obtain

\begin{cor}\label{cor_dist_proj_lip}Let $(X,d,\mu)$ be an RCD$(1,\infty)$ space admitting $CD(1,\infty)$ disintegration, and whose spectral gap satisfies $\lambda_1 \leq 1 + \epsilon$. Let $g$ be the guiding function as in \eqref{gf}. Then,
$$W_1(g_{\sharp}\mu, \gamma) \leq C(\theta)\epsilon^{1/20-\theta}$$
$$d_{TV}(g_{\sharp}\mu, \gamma) \leq 2C(\theta)\epsilon^{1/20-\theta}$$
\end{cor}

We shall use the estimate on the total variation distance to get an estimate on the observable diameter via the following lemma: 

\begin{lem}\label{lem_obs_tv}
Assume that $\mu$ and $\nu$ are two probability measures on a metric space $(X,d)$ such that $d_{TV}(\mu, \nu) \leq \epsilon$, where $\epsilon>0$. Then for any $\kappa > \epsilon$, the following inequality holds
$$D_{obs}(\mu; \kappa - \epsilon) \geq D_{obs}( \nu; \kappa).$$
\end{lem}

\begin{proof}
We shall argue by contradiction. Fix $\kappa$, and assume that $D_{obs}(\mu; \kappa - \epsilon) < D_{obs}(\nu; \kappa) - t$ for some $t > 0$. Then for any $f$ $1$-Lipschitz there exists $E \in \R$ such that $\mu(E) \geq 1 - \kappa + \epsilon$ and $\operatorname{diam}(f(E)) \leq D_{obs}(\nu; \kappa) - t$. But then $\nu(E) \geq 1 - \kappa + \epsilon - d_{TV}(\mu, \nu) \geq 1 - \kappa$, and therefore, taking the sup over all $f$, we get $D_{obs}(\nu; \kappa) \leq D_{obs}(\nu; \kappa)-t$ hence a contradiction. 
\end{proof}

Since the function $g$ in Corollary \ref{cor_dist_proj_lip} is $1$-Lipschitz, by definition of observable diameter, we infer
 $$D_{obs}(\mu; \cdot) \geq D_{obs}( g_{\sharp}\mu, \cdot)$$
 on $(0,1)$ (see \cite{Shioya} for a proof). Therefore, by combining Theorem \ref{ObsSep}, Corollary \ref{cor_dist_proj_lip}, and Lemma \ref{lem_obs_tv}, we get the proof of Theorem \ref{sg_to_diam}

\subsection{Proof of Theorem \ref{thm_dobs_to_sg}}

We assume throughout this part that

\begin{equation}\label{amod}
D_{obs}((M,d,\mu); \kappa) \geq D_{obs}((\R, |\cdot|, \gamma); \kappa)     -\epsilon
\end{equation}
% & \leq   Sep((M,d,\mu); \kappa/2)
for a given $0<\kappa<1$ and $\epsilon>0$ small enough. Our aim is to show  Theorem \ref{thm_dobs_to_sg}, namely that such an RCD-space has almost minimal spectral gap.

Let us denote by $V_s(A)$ the $s$-open neighborhood of a subset $A$ of a given metric space ($(\R,|\cdot|)$ in most cases).

As a consequence of Theorem \ref{ObsCompa}, we get 

\begin{equation}\label{amsd} 
Sep((M,d,\mu); \kappa/2) \geq Sep((\R,|\cdot|,\gamma);\kappa/2) -\epsilon.
\end{equation}
Instead of the observable diameter, we shall mainly deal with the separation distance through \eqref{amsd}. The proof mostly relies on a careful study of the 1D case. We will then infer the bound on the spectral gap by means of the disintegration property. In what follows, $(\R, |\cdot|, m)$ stands for an $RCD(1,\infty)$-space, namely $m=e^{-\phi(x)}\, dx$ is a probability measure such that $\phi(x) - x^2/2$ is convex.

 \subsubsection{The $1D$ case}
 
In this part we assume that  $(\R, |\cdot|, m)$ is an RCD$(1,\infty)$-space such that for some $\theta>0$ and for $\eta>0$ sufficiently small

\begin{equation}\label{amsd1}
 D_{obs}((\R, |\cdot|, m); \theta) \geq D_{obs}((\R, |\cdot|, \gamma); \theta)     -\eta.
 \end{equation}
 
Let  us start with some basic properties regarding the Gauss measure. Recall that $\sigma(x) = \gamma(]-\infty, x])$.

\begin{lem}For $0<\kappa_1 \leq a<b < 1/2$ where $\kappa_1 $ is fixed and sufficiently small, the following estimates hold
 \begin{equation}\label{sigEst}
 {\sqrt{2\pi}}(b-a) \leq \sigma^{-1}(b)- \sigma^{-1}(a) \leq \frac{\sqrt{2\pi}}{\kappa_1^2}(b-a).
\end{equation}
\end{lem}

\begin{proof}
We denote by $e^{-\phi_{\gamma}}$ the density of the standard Gauss measure on $\R$. By definition of $\sigma^{-1}$, we have
$$ \int_{\sigma^{-1}(a)}^{\sigma^{-1}(b)} e^{-\phi_{\gamma}(s)}\,ds = b-a.$$
The first inequality follows readily from this equality. For the second one, we need to estimate $|\sigma^{-1}(\kappa_1)|$ from above. We write for $\kappa_1 < 1/2$
\begin{eqnarray*} 
\kappa_1 & = & \int_{-\infty}^{\sigma^{-1}(\kappa_1)} e^{-\phi_{\gamma}(s)} \, ds \\
         & \leq & \frac{1 }{\sqrt{2\pi}|\sigma^{-1}(\kappa_1)|} e^{-(\sigma^{-1}(\kappa_1))^2/2}.
\end{eqnarray*}
From the inequality above, we infer $|\sigma^{-1} (\kappa_1)| \leq 2 \sqrt{-\ln \kappa_1}$  as soon as $\sigma^{-1}(\kappa_1) \leq -1$, which in turn implies the needed lower bound on $e^{-\phi_{\gamma}(s)}$  on $[\sigma^{-1}(a), \sigma^{-1}(b)]$ for $\kappa_1 < \gamma(]-\infty, -1])$.

\end{proof}

\begin{rque}\label{EstSep}Recall that $Sep((\R,|\cdot|,\gamma);\theta/2) = 2|\sigma^{-1} (\theta/2)|$. Thus, the preceding lemma also gives bounds on the separation distance on the Gauss space .
\end{rque}

Easy computations assert that the interval $[\sigma^{-1} (\theta/2), -\sigma^{-1} (\theta/2)]$ is the only closed interval of given length $R:= 2|\sigma^{-1} (\theta/2)|$ whose $\gamma$-measure is maximal, namely $1-\theta$. More precisely, there exists a constant $C(R)$, depending continuously on $R>0$, such that

\begin{equation}\label{stabObsGa}
\left| \gamma([s-R/2,s+R/2])  -(1-\theta)\right| \geq C(R)\, s^2,
\end{equation}
for any $s$ sufficiently small (say, such that $  \gamma([s-R/2,s+R/2]) \geq 1/2$).

 \begin{lem}\label{closemass} Let $(\R, |\cdot|, m)$ be a $CD(1,\infty)$-space.  Let $A_1,A_2\subset \R $ be two Borel subsets. % such that $\max \{ m(A_1), m(A_2)\} <1/2$ and $d(A_1,A_2) >0$. 
 Then, the following inequality holds
$$ d(A_1,A_2) \leq -\sigma^{-1}(m(A_1)) - \sigma^{-1}(m(A_2)).$$
In particular, assuming $\min \{ m(A_1), m(A_2)\} \geq \theta/2>0$ and 
$$d(A_1,A_2) \geq Sep((\R,|\cdot|,\gamma);\theta/2) -2\eta>0,$$ 
the following inequality holds
$$ \max\{ m(A_1), m(A_2) \} \leq \theta/2 + \sqrt{2/\pi}\, \eta.$$

\end{lem}

\begin{proof}

The proof is a simple consequence of an integrated version of the isoperimetric inequality: given a Borel subset $A$ of $\R$ and $t \geq 0$,
\begin{equation}\label{eII}
m(V_t(A)) \geq \gamma((-\infty, \sigma^{-1}(m(A)) +t]).
\end{equation}
In particular, for $r = d(A_1,A_2)$, %it reads $\gamma(\gamma((-\infty, \sigma^{-1}(\mu(A_1)) +R]) \leq m(V_R(A_1)$.
 \begin{eqnarray*}
\gamma((-\infty, \sigma^{-1}(m(A_1)) +r]) &  \leq & m(V_r(A_1)) \\
               & \leq & 1 -m(A_2), \\
 			  & \leq & \gamma((-\infty,|\sigma^{-1}(m(A_2))|]),
 	%		  & \leq &\gamma ((-\infty, \sigma^{-1}(\mu(A_1)) + Sep((\R,|\cdot|,\gamma);\kappa/2)),
 \end{eqnarray*}
where the second inequality follows from $A_2 \cap V_r(A_1) =\emptyset$. The inequality relating the expressions at the ends implies
$$ \sigma^{-1}(m(A_1)) +r \leq -\sigma^{-1}(m(A_2))$$
and the first inequality is proved.

To prove the second statement let us rewrite the first inequality as
\begin{equation}\label{es}d(A_1,A_2) \leq 2|\sigma^{-1} (\theta/2)| + (\sigma^{-1}(\theta/2) - \sigma^{-1}(m(A_1)) + (\sigma^{-1}(\theta/2) - \sigma^{-1}(m(A_2)).
\end{equation}
We are left with estimating $ (\sigma^{-1}(\theta/2) - \sigma^{-1}(m(A_i))$ from above. According to \eqref{sigEst}, we have
$$  \sigma^{-1}(m(A_i))  - \sigma^{-1}(\theta/2) \geq \sqrt{2\pi} (m(A_i) -\theta/2).$$
 By combining \eqref{es} together with the above inequality, we get the second claim.

\end{proof}

\begin{rque}\label{remclosemass}
The proof above applies verbatim to general RCD$(1,\infty)$-spaces since it is mainly based on the isoperimetric inequality stated in Theorem \ref{IIn}. It gives the same results.
\end{rque}
In the same vein, we shall also need the following estimate: 

\begin{lem}\label{uppboundmass}Let $(\R,|\cdot |,m)$ be a CD$(1,\infty)$ space, and let $A_1,A_2\subset \R $ be two Borel subsets such that $\min \{ m(A_1), m(A_2)\} \geq \theta/2>0$ and 
$$r:= d(A_1,A_2) \geq Sep((\R,|\cdot|,\gamma);\theta/2) -2\eta>0.$$

Then, for any $s \in [0, r]$ and $i \in \{1,2\}$
$$ m((V_s(A_i)) \leq  \gamma((-\infty,\sigma^{-1}(m(A_i))+s]) + \sqrt{2/\pi}\,\eta. $$
\end{lem}

\begin{proof}

We argue as in the proof of the previous lemma. Set $i=1$.
 \begin{eqnarray*}
 m(V_s(A_1))  & \leq &1 -m(V_{r-s}(A_2)) , \\
 		&\leq	& 1 - \gamma( [ -\sigma^{-1}(m(A_2))-r +s, +\infty)) \\
 			  & \leq & \gamma((-\infty, -\sigma^{-1}(m(A_2)) +2\sigma^{-1}(\theta/2) +2\eta +s]), \\
 			   & \leq & \gamma((-\infty, \sigma^{-1}(\theta/2)) +s + 2\eta]),
 			  \end{eqnarray*}
where we use the symmetry of $\gamma$, the inequality on $r=d(A_1,A_2)$ to infer the inequality on the third line, and $m(A_1) \geq \theta/2$ to get the estimate in the last one. The proof then follows from $ \gamma((c, c+\delta) )\leq \frac{1}{\sqrt{2\pi}} \delta,$ where $c$ is arbitrary.

\end{proof}

Let us fix some notation. Let $A_1$ and $A_2$ as in Lemma \ref{uppboundmass}. 
Without loss of generality, one can assume that the interval $[a_-,a_+]$ is such that $(a_-,a_+,) \cap (A_1 \cup A_2) =\emptyset$ and $a_+ -a_- = d(A_1,A_2)$. Up to reversing $A_1$ and $A_2$, we can further assume that $a_- \in \overline{A_1}$ and $a_+ \in \overline{A_2}$.

Our next goal is to show that $\phi$ is uniformly close to $\phi_{\gamma}$ on $[a_-,a_+]$. To this aim, we shall use properties of the optimal map $T$ such that $T_{\sharp} \, \gamma = m$. Let us recall that $T$ is non-decreasing and $1$-Lipschitz, as a consequence of Caffarelli's contraction theorem \cite{Caf00, FGP20}.

We set 
\begin{equation}\label{defalpha}
\alpha_-:= \sigma^{-1}(m(A_1)) \mbox{ and } \alpha_+:= \sigma^{-1}(1-m(A_2))= -\sigma^{-1}(m(A_2)),
\end{equation}  in other terms $T(\alpha_{\pm})= a_{\pm}$.

\begin{prop}\label{UnifEsti} Let $(\R, |\cdot|, m=e^{-\phi(s)}ds)$ be a CD$(1,\infty)$ space such that for $\theta>0$ and $0<\eta <\eta_0(\theta)$,
$$   d(A_1,A_2) \geq Sep((\R,|\cdot|,\gamma);\theta/2) -2\eta>0.$$
Then, up to translating $m$ so that the median is at the origin, there exists $C=C(\theta)>0$ such that for any $s \in [a_-,a_+]$,
$$ |\phi(s) - \phi_{\gamma} (s) | \leq  C\eta^{1/2}. $$
\end{prop}
\begin{proof}
In this proof, $C(\theta)$ will be a constant that only depends on $\theta$, but which may change from line to line. Set $\delta_+:= a_+-\alpha_+$ and $\delta_-:=a_- -\alpha_-$. We first claim that
\begin{eqnarray*}
 \phi (s) &\leq & \phi_{\gamma} (s - \delta_+) \; \; \mbox{ on } [s_l, a_+]  \\
  \phi (s) & \leq & \phi_{\gamma} (s - \delta_-) \; \; \mbox{ on } [a_-, s_l],
\end{eqnarray*}
where $s_l = \rm{argmin} \, \phi$ (recall $\phi$ is $1$-convex).
We shall only prove the first inequality, the second one is proved in the same way. 

Setting the median of $m$ at the origin is equivalent to enforcing $T(0) = 0$, since the median of the Gaussian measure is also at the origin. Thanks to the monotonicity of $T$ and the Lipschitz bound recalled above, we notice that $ T(s) \geq s + \delta_+$ on $(-\infty,\alpha_+].$ By combining this together with the fact
that $e^{-\phi}$ is decreasing on $[s_l,+\infty)$, we get on $[s_l-\delta_+,\alpha_+]$ 
$$ e^{-\phi(T(s))} \leq e^{-\phi(s + \delta_+)}.$$
By definition of $T$, we have for almost every $s$, 
$$e^{-\phi_{\gamma}(s)}=e^{-\phi(T(s))} T'(s)\leq e^{-\phi(T(s))}, $$
 and the proof of the inequality is complete by continuity of $\phi $ and $\phi_{\gamma}$.

Next, we prove that $\delta_+ \simeq \delta_-\simeq T(0) = 0$. Using that $s \mapsto T(s) -s$ is non-increasing on $[\alpha_-,\alpha_+]$, we get that 
$$  0\leq   (  T(\alpha_-) -\alpha_-)  - (  T(\alpha_+)) -\alpha_+   )=  \delta_- - \delta_+ . = (\alpha_+-\alpha_-) -(a_+-a_-) $$
By assumption on $d(A_1,A_2) = a_+ -a_-$ and using that $-\sigma^{-1}$ is non-increasing, we obtain
\begin{eqnarray*}
a_+ -a_- & \geq &  Sep((\R,|\cdot|,\gamma);\theta/2) -2\eta \\
         & \geq & -2\sigma^{-1}(\theta/2) -2\eta \\
         & \geq & -\sigma^{-1} (m(A_1))  -\sigma^{-1} (m(A_2)) -2\eta \\
         & \geq & \alpha_+-\alpha_- -2\eta,
\end{eqnarray*}
where we use \eqref{defalpha} to get the last inequality. Therefore, we obtain
\begin{equation}\label{deltaEst}
\left\{  \begin{array}{rcl} & 0 & \leq  \delta_- -\delta_+ \leq 2\eta  \\
\mbox{ and }  & \delta_+ & \leq  T(0) = 0 \leq \delta_-
\end{array}\right.
\end{equation}  by monotonicity. As a consequence, $|\delta_+| + |\delta_-| \leq 2\eta$. 

% These properties also imply
% $$  \left| \int_{[s_l-\delta_-,s_l-\delta_+]}  e^{-\phi_{\gamma}(s)} \,ds  \right|  \leq \sqrt{2/\pi}\, \eta.$$
 
%We now aim to prove that, for $\eta$ small enough,  $|T(0)|\leq C(\theta)$. 
Now recall that by definition of $\alpha_{\pm}$,  $\min (\gamma ((-\infty,\alpha_-]), \gamma([\alpha_+,+\infty))) \geq \theta/2$. Moreover, Lemma \ref{closemass} implies
\begin{equation}\label{EstiAlph} \max (\gamma ((-\infty,\alpha_-]), \gamma([\alpha_+,+\infty))) \leq \theta/2 + \sqrt{2/\pi} \, \eta.
\end{equation}
Hence, \eqref{stabObsGa} implies the existence of $C(\theta)>0$ such that
\begin{equation}\label{EstiAlphII} |(\alpha_+  + \alpha_-)|/2 \leq C(\theta) \sqrt{\eta},
\end{equation}

By combining this together  with \eqref{deltaEst}, we infer
\begin{equation}\label{ApproMid} \left|\frac{a_++a_-}{2}\right| \leq C \sqrt{\eta}.
\end{equation}
As a consequence, 
\begin{align}\label{EstiPhi}
 \phi (s) &\leq \phi_{\gamma} (s -\delta_+)  \notag \\
&= \phi_\gamma(s) -\delta_+ s + \delta_+^2/2 \notag \\
&\leq \phi_\gamma(s) + C(\theta)\eta
\end{align} 
on $[a_-,a_+] \subset [-C(\theta), C(\theta)]$.

The last step of the proof relies on Markov's inequality and the fact that $\phi-\phi_{\gamma}$ is a convex function. More precisely, given $\nu >0$, we set
$$S_{\nu}:= \{ s \in [a_-,a_+]\, | \phi(s) \geq \phi_{\gamma}(s) -\nu \}.$$
Observe that $(a_-,a_+) \cap (A_1 \cup A_2)=\emptyset$ implies
\begin{equation}\label{EstiInt4.11} \int_{[a_-,a_+]} e^{-\phi(s)} \,ds  \leq  1 - m(A_1) -m(A_2)=  \int_{[\alpha_-,\alpha_+]}  e^{-\phi_{\gamma}(s)} \,ds.
 \end{equation}
 Besides,  \eqref{EstiAlph} implies the existence of $C>0$ such that 
 $$ \int_{[a_-,a_+]} e^{-\phi_{\gamma} (s)} \, ds \geq  \int_{[\alpha_-,\alpha_+]} e^{-\phi_{\gamma} (s)}\, ds -C\eta.$$

Using \eqref{EstiPhi}, we have
\begin{eqnarray*}
\int_{[a_-,a_+]} e^{-\phi(s)} \,ds  & \geq & e^{\nu}\, \int_{S_{\nu}^c}  e^{-\phi_{\gamma}(s )} \,ds  + e^{-C \eta}\, \int_{S_{\nu}}   e^{-\phi_{\gamma}(s)} \,ds   \\
        & \geq & \big(e^{\nu}-1\big) \, \int_{S_{\nu}^c}  e^{-\phi_{\gamma}(s)} \,ds  + e^{-C \eta}\left( \int_{[\alpha_-,\alpha_+]}  e^{-\phi_{\gamma}(s)} \,ds  -C\, \eta\right).
\end{eqnarray*}
This estimate and \eqref{EstiInt4.11} yield
$$ \gamma(  S_{\nu}^c ) \leq C \, \frac{\eta}{e^{\nu}-1} \leq C \, \frac{\eta}{\nu},$$
which suggests to choose $\nu= {\eta}^{1/2}$.

Now, the convexity of $h:=\phi-\phi_{\gamma}$  forces $S_{\nu}^c$ to be an interval, and the Lebesgue measure of $S^c_{\eta^{1/2}}$ is bounded from above by $C(\theta) \eta^{1/2}$. Thales' theorem and the convexity of $h$ then imply  $\min h \geq -{C}(\theta) {\eta}^{1/2}$ and the proof is complete.

\end{proof}

Let us set $a_{\eta}:= |\sigma^{-1}(\sqrt{\eta})|$.  Our goal is to extend the estimate on $|\phi-\phi_{\gamma}|$ from $[a_-,a_+]$ to a much larger interval (depending on $\eta$). This is the content of the next lemma.

\begin{lem}\label{EstPhiII} Let $(\R, |\cdot|, m=e^{-\phi(s)}ds)$ be a CD$(1,\infty)$ space such that for $\theta>0$ and $0<\eta <\eta_0(\theta)$
$$  Sep((\R,|\cdot|,m);\theta/2) \geq Sep((\R,|\cdot|,\gamma);\theta/2) -2\eta>0.$$
Then, up to translating $m$ so that the median is at the origin, there exists $C=C(\theta)>0$ and $b_\eta > 0$ such that, for any $s \in [-b_\eta,b_\eta]$,
$$ |\phi(s) - \phi_{\gamma} (s) | \leq  C\eta^{1/10}, $$
and $m((-\infty, -b_\eta]) \leq \gamma((-\infty, -b_\eta]) \leq C\eta^{1/10}$.
\end{lem}
 
 \begin{proof}
First of all, by definition of $a_\eta$, $\gamma((-\infty, -\aet]) = \sqrt{\eta}$, and since the median of $m$ is at the origin, the Gaussian concentration inequality (which is the integrated form of the isoperimetric inequality for RCD$(1, \infty)$ spaces states that $m((-\infty, a_\eta]) \leq \gamma((-\infty, -a_\eta]) $. 

 Set $h  :=\phi -\phi_{\gamma}$.  Recall that $|h|\leq C\eta^{1/2}$ on $[a_-,a_+]$ and that $h$ is convex. Markov's inequality yields the existence of $\hat{a}_+ \in [a_-,a_+] $ such that
$$ |\phi'(\hat{a}_+) -\phi_{\gamma}'(\hat{a}_+) | \leq C \eta^{1/4}   \mbox{  and  } |a_+ -\hat{a}_+| \leq \eta^{1/4}.$$
For $s \geq \hat{a}_{+}$, the $1$-convexity of $\phi$ yields
$$ \phi(s) \geq \phi(\hat{a}_{+}) + \phi'(\hat{a}_{+}) (s-\hat{a}_{+}) +  1/2(s-\hat{a}_+)^2.$$
Noticing that $\phi_{\gamma}(\hat{a}_+)+ \phi_{\gamma}'(\hat{a}_+) (s-\hat{a}_+)  +  1/2(s-\hat{a}+)^2 = \phi_{\gamma}(x),$ we infer from the previous inequality:
$$ \phi(s) \geq \phi_{\gamma}(s) + E(s),$$
where $E(s) = -C {\eta}^{1/4} (1 + (s-\hat{a}_+)).$

\noindent Recall that $\aet = |\sigma^{-1}(\sqrt{\eta})| \sim \sqrt{\ln (1/\eta)}$ for $\eta \sim 0$ thus, for $\eta$ small enough, there exists a constant $C$ such that 
\begin{equation}\label{EtEsti}
\aet \leq C\sqrt{\ln (1/\eta)},
\end{equation}
and there exists a constant $C$ such that
 $$E(s) \geq -C {\eta}^{1/4} (1 + a_\eta) \geq -C \eta^{1/5}$$
on $[\hat{a}_{+}, \aet]$. Consequently, 
 
 \begin{equation}\label{UppEstEphi}
 \phi \geq \phi_{\gamma} - C \eta^{1/5 }
\end{equation}  
on $[{a}_{+}, \aet] \subset [\hat{a}_{+}, \aet]$.
 %  \int_{[\hat{a}_+, \aet]} e^{-\phi(s)}\, ds \leq e^{C\eta^{1/5}}\, \int_{[\hat{a}_+, \aet]} e^{-\phi_{\gamma} (s)} \, ds.

The same method applies on $[-\aet, \hat{a}_-]$ (where the definition of $\hat{a}_-$ is clear from the context) %defined in a similar fashion as $\hat{a}_+$)
  and leads to the same estimate $ \phi \geq \phi_{\gamma} - C \eta^{1/5}$.

We then have, using the estimates on the interval $[a_-, a_+]$ and that $|a_+ - \hat{a}_+| \leq  {\eta}^{1/4}$, 
\begin{align*}
\left| \int_{[-\aet,\aet]\setminus [{a}_-,{a}_+]} e^{-\phi(s)} -e^{-\phi_{\gamma}(s)} \, ds\right|&= \left|\gamma([{a}_-, {a}_+]) - m([{a}_-, {a}_+]) \right.\\ 
& \hspace{2cm} \left. + \gamma([-\aet, \aet]^c) - m([-\aet, \aet]^c)\right| \\
&\leq 2\gamma([\aet, + \infty)) + m((-\infty, -\aet]) \\
& \hspace{2cm} + m([\aet, + \infty)) + C\sqrt{\eta} \\ 
&\leq C\sqrt{\eta} .
\end{align*}

As in the proof of Proposition \ref{UnifEsti}, Markov's inequality and the estimates above imply that
$$S_{\eta}:=\{s\in  [-\aet,\aet]\setminus [{a}_-,{a}_+]| \,(\phi-\phi_{\gamma})(s) \leq \eta^{1/10}\}$$
satisfies
$$\gamma (S_{\eta}^c) \leq C \eta^{1/10}.$$

Finally, the convexity of $\phi-\phi_{\gamma}$ implies that $S_{\eta}$ is the union of two intervals; besides, thanks to Proposition \ref{UnifEsti}, $S_{\eta}$ contains a neighborhood of ${a}_-$ and ${a}_+$. Therefore $S_{\eta}^c$ is the union of two intervals $[-\aet,-\bet] \cup [\bet,\aet]$. %of length of order (at most) $\eta^{1/10}$.
We also have
$$m((-\infty, -b_\eta]) \leq \gamma((-\infty, -b_\eta]) \leq \gamma((-\infty,-\aet]) + \gamma ((-\aet,-\bet]) \leq C\eta^{1/10}.$$

 \end{proof}
 
 We are left with estimating the variance of the identity map relative to the measure $m$.

 \begin{lem}\label{VarEsti1D} Let $(\R, |\cdot|, m=e^{-\phi(s)}ds)$ be a CD$(1,\infty)$ space such that for $\theta>0$ and $0<\eta <\eta_0(\theta)$,
$$  Sep((\R,|\cdot|,m);\theta/2) \geq Sep((\R,|\cdot|,\gamma);\theta/2) -2\eta>0.$$
 Then, up to translating $m$ so that its median is at the origin, the following estimates hold
 
 $$ \left| \int_{\R} s \, dm(s)\right| \leq C \eta^{1/11},$$
 and
 $$ \int_{\R} s^2 \, dm(s) \geq 1 -C \eta^{1/11}.$$ 
 In particular, 
 $$ Var_m(s) \geq 1 -C \eta^{1/11}.$$
 
 \end{lem}

 \begin{proof}
Let us start with some estimates relative to the Gaussian measure. Using $$\frac{e^{-\phi_{\gamma}}(\bet)}{\bet}\sim \int_{\bet}^{+ \infty} e^{-\phi_{\gamma}(s)} \, ds \leq C\eta^{1/10},$$
we infer, for $\eta$ small enough,

$$ \int_{\bet}^{+ \infty} s\, e^{-\phi_{\gamma}(s)} \, ds = e^{-\phi_{\gamma} (\bet)} \leq  C \bet \eta^{1/10} \leq C\eta^{1/10} \sigma^{-1}(\eta^{1/10}) \leq C\eta^{1/11}.$$
Moreover, thanks to 
$ \int_{\bet}^{+ \infty} s^2\, e^{-\phi_{\gamma}(s)} \, ds = \bet e^{-\phi_{\gamma}(\bet)} + \int_{\bet}^{+ \infty} \, e^{-\phi_{\gamma}(s)} \, ds$, we get
$$ \int_{\bet}^{+ \infty} s^2\, e^{-\phi_{\gamma}(s)} \, ds \leq  C\eta^{1/11}. $$

Now, we use Lemma \ref{EstPhiII} to infer similar bounds for $m$. More precisely,

\begin{eqnarray*} \left|\int_{[-\bet,\bet]} s\, e^{-\phi(s)} \, ds - \int_{[-\bet,\bet]} s\, e^{-\phi_{\gamma}(s)} \, ds \right| & \leq & \int_{[-\bet,\bet]} |s| \big(e^{|\phi(s)-\phi_{\gamma}(s)|}-1\big) e^{-\phi_{\gamma} (s)} \, ds \\
& \leq & \int_{[-\bet,\bet]}|s| \big(e^{C \eta^{1/10}}-1 \big)e^{-\phi_{\gamma} (s)} \, ds \\
  & \leq & 2 (1 -e^{-\phi_{\gamma}(\bet)}) {C \eta^{1/10}} \leq C \eta^{1/10}.
 \end{eqnarray*}   
We also have
 \begin{eqnarray*} \left|\int_{[-\bet,\bet]}s^2 e^{-\phi(s)} \, ds - \int_{[-\bet,\bet]} s^2 e^{-\phi_{\gamma}(s)} \, ds \right| & \leq & \int_{[-\bet,\bet]} s^2\big( e^{C \eta^{1/10}}-1\big) e^{-\phi_{\gamma} (s)} \, ds \\
   & \leq & C \eta^{1/10}.
 \end{eqnarray*}

Last, we claim that $0 \leq T(s) \leq s $ (resp. $ 0 \geq T(s)\geq s $) on $[\bet,+\infty) \subset [\alpha_+,+\infty)$ (resp. on $(-\infty,-\bet]$), since $T$ is monotone and $1$-Lipschitz and $T(0) = 0$.

With this property, we can estimate the remaining part of the expectation of $m$:
\begin{eqnarray*} \int_{[\bet, +\infty]} s\, e^{-\phi(s)} \, ds &\leq &
  \int_{[T^{-1}(\bet), +\infty)}  T(u) \, e^{-\phi_{\gamma}(u)} \, du \\
 & \leq & \int_{[\bet, +\infty)} u e^{-\phi_{\gamma}(u)} \, du \\
 & \leq & C\eta^{1/11},
\end{eqnarray*}
where we use the upper tail bound for the Gaussian measure. The same method applies on $(-\infty,-\bet]$, and the proof is complete.
 \end{proof}
 
%%%%%%% General Case.

\subsubsection{The general case}
According to \eqref{amsd}, there exists sets $A_1$ and $A_2$ such that
$$ d(A_1,A_2) \geq Sep((\R,|\cdot|,\gamma);\kappa/2) -2\epsilon,$$
and $\min (\mu(A_1), \mu(A_2)) \geq \kappa/2$.

We now introduce a needle decomposition relative to a function $f$ defined in terms of the sets $A_1$ and $A_2$ above. Precisely, we set
\begin{equation}\label{ndf}
f= \chi_{A_1} -\chi_{A_2} - \mu(A_1) + \mu(A_2).
\end{equation}
This function satisfies

$$ \int |f(x)| d(x,x_0)\, d\mu< \infty \mbox{         and       } \int f\, d\mu =0.$$

Applying the needle decomposition associated to $f$ defined on  $(M,d,\mu)$, we get that for $\Q$ a.e $q$, $\int_{X_q} f \, dm_q =0$. This yields for a.e $q$

$$ m_q(A_1\cap X_q) - m_q (A_2 \cap X_q) = \mu(A_1) - \mu(A_2).$$

Therefore, according to Lemma \ref{closemass} (see Remark \ref{remclosemass}), we get
\begin{equation}\label{mqE}
|m_q(A_1\cap X_q) - m_q (A_2 \cap X_q)|\leq \sqrt{2/\pi}\, \epsilon.
\end{equation} 

From this estimate, we infer the following result

\begin{lem}
Let $(M,d,\mu)$ be an RCD$(1,\infty)$ space admitting CD$(1,\infty)$ disintegration and such that \eqref{amod} holds. Then, relatively to the needle decomposition associated to $f$ as in \eqref{ndf}, there exists a set of needles $Q_{\delta}$ such that $\Q(Q_{\delta}) \geq 1 -{\delta}$, and for any needle $q \in Q_{\delta}$, we have
\begin{equation}\label{amod1d}
\min\{ m_q (A_1), m_q(A_2)\} \geq \kappa/2 - \frac{\sqrt{8/\pi} \,\epsilon}{\delta}-  \sqrt{2/\pi}\, \epsilon.
\end{equation} 
\end{lem}

\begin{proof}
Since we assume there is a needle decomposition, there exists a partition of $M$: $M= \Tau \sqcup Z$ such that $\mu|_{\Tau}$ admits a needle decomposition and $f=0$ $\mu$-a.e. on $Z$. Now, by the very definition of $f$, $Z \subset (A_1 \cup A_2)^c$ (actually $Z=\emptyset$ whenever $\mu(A_1)\neq \mu(A_2)$). Therefore, we obtain
\begin{equation}\label{mma} \mu(A_i) = \int_{\Tau} \chi_{A_i} \, d\mu = \int_Q m_q(A_i) \, d\Q.
\end{equation}
Recall that each needle is in particular a geodesic, thus the distance between $A_1\cap X_q$ and $A_2 \cap X_q$ along $X_q$ is greater or equal to $d(A_1,A_2)$, where 
$$d(A_1,A_2) \geq Sep((\R,|\cdot|,\gamma);\kappa/2) -2\epsilon> Sep((\R,|\cdot|,\gamma);\kappa/2 + \sqrt{2/\pi} \epsilon)$$
according to \eqref{sigEst} and Remark \ref{EstSep}. Using Theorem \ref{ObsCompa}, we infer
$$ d(A_1,A_2) > Sep((X_q,|\cdot|,m_q);\kappa/2 +\sqrt{2/\pi} \epsilon).$$
This inequality yields 
$$ \min\{ m_q (A_1), m_q(A_2)\} \leq \kappa/2 + \sqrt{2/\pi} \epsilon,$$
and we can replace $\min$ by $\max$ in the above equation thanks to \eqref{mqE}, provided that $\sqrt{2/\pi} \epsilon$ is replaced by $2\sqrt{2/\pi} \epsilon$.

The end of the proof consists in applying Markov's inequality in \eqref{mma} combined with $\max\{ m_q (A_1), m_q(A_2)\} \leq \kappa/2 + \sqrt{8/\pi} \epsilon$.

\noindent Defining $Q_{\delta}:=\{ q \in Q|\, m_q (A_1) \geq \kappa/2 -\frac{\sqrt{8/\pi} \,\epsilon}{\delta}\}$, we get 
\begin{eqnarray*}
\mu(A_1)& \leq & \int_{Q_{\delta}^c} m_q(A_1)\, d\Q + \int_{Q_{\delta}} m_q(A_1)\, d\Q \\
          & < & (\kappa/2 -\frac{\sqrt{8/\pi} \epsilon}{\delta}) \Q(Q_{\delta}^c) + (\kappa/2 + \sqrt{8/\pi} \epsilon) ( 1 - \Q(Q_{\delta}^c)),
\end{eqnarray*}
from which we infer (since $\mu(A_1)\geq \kappa/2$)
$$ \Q(Q_{\delta}^c ) \leq \delta.$$
Applying \eqref{mqE} yields the result.

\end{proof}

 For a.e. $q \in Q_{\sqrt{\epsilon}}$, there exists a constant $\widehat{C}>0$ such that  $$\min\{m_q(A_1), m_q(A_2)\} \geq \kappa/2 -\widehat{C}\,\sqrt{\epsilon},$$
while, as explained in the proof above, for almost every $q$, 
$$d(A_1\cap X_q, A_2\cap X_q) \geq d(A_1,A_2) \geq  Sep((X_q,|\cdot|,m_q);\kappa/2) -2\epsilon.$$
Besides, using  \eqref{sigEst} and Remark \ref{EstSep}, we obtain
$$  Sep((X_q,|\cdot|,m_q);\kappa/2) -2\epsilon  \geq  Sep((X_q,|\cdot|,m_q);\kappa/2 -\widehat{C}\sqrt{\epsilon}) - C \sqrt{\epsilon}.$$ 

According to Lemma \ref{VarEsti1D}, this implies for a.e $q \in Q_{\sqrt{\epsilon}} $
the existence of a constant $C=C(\kappa)>0$ such that
\begin{eqnarray}\label{VarEstII}
Var_{m_q} (s) &\geq & 1 -C\epsilon^{1/22}, \nonumber \\
\mbox{and } \hspace{0,8cm} \Q(Q_{\sqrt{\epsilon}}) &\geq & 1 -\sqrt{\epsilon}.
\end{eqnarray}
%& \mbox{and} & \nonumber \\

\subsubsection{Proof of Theorem \ref{thm_dobs_to_sg}}

 The proof builds upon the decomposition relative to $f$ as in \eqref{ndf}. According to Theorem \ref{thm_needle_dec}, there exists a guiding function $u$ which is $1$-Lipschitz, and such that for a.e. $q \in Q$, $u(s)=s +c_q$ on $X_q$ (with $c_q \in \R)$. As a consequence, we have 
 $$ \int |\nabla u|^2 \, d\mu \leq 1,$$
 and
 \begin{eqnarray*}
  Var_{\mu} (u) & = & \int_X u^2 -\left(\int_X u \,d\mu\right)^2 \, d\mu \\
  				& \geq & \int_Q \int_{X_q} u^2 \, dm_q \,d\Q(q) - \int_Q \left(\int u\, dm_q\right)^2 \,d\Q(q) \\
  				& = & \int_Q Var_{m_q} (u)\, d\Q(q),
 \end{eqnarray*}
where the inequality follows from Cauchy-Schwarz. Using \eqref{VarEstII}, we then get
$$  Var_{\mu} (u) \geq \Q(Q_{\sqrt{\epsilon}}) \big(1 - C\epsilon^{1/22}\big) \geq 1 - C\epsilon^{1/22} \geq (1 - C\epsilon^{1/22})\int{|\nabla u|^2d\mu},$$
which gives the upper bound on the spectral gap.

\vspace{3mm}

\underline{\textbf{Acknowledgments}}: We thank Fabio Cavalletti, Thomas Courtade, Guido De Philippis, Bo'az Klartag, Michel Ledoux and Shin-Ichi Ohta for valuable discussions. M.F. was supported by ANR-11-LABX-0040-CIMI within the program ANR-11-IDEX-0002-02, as well as ANR projects EFI (ANR-17-CE40-0030) and MESA (ANR-18-CE40-006). J.B. was supported by ANR-17-CE40-0034. 

\bibliographystyle{plain}

\end{document}